\documentclass[12pt]{article}
\usepackage{amsmath,amsthm,amsfonts}
\usepackage{amscd}
\usepackage[active]{srcltx}
\title{Topological classification of M\"obius transformations}

\author{Tetiana Rybalkina\qquad
Vladimir V. Sergeichuk\\
rybalkina\_t@ukr.net\qquad
sergeich@imath.kiev.ua\\
Institute of Mathematics, Kiev, Ukraine}

\newtheorem{theorem}{Theorem}[section]
\newtheorem{lemma}{Lemma}[section]
\newtheorem{corollary}{Corollary}[section]

\DeclareMathOperator{\tr}{tr}
\DeclareMathOperator{\size}{size}

\begin{document}

\date{}
\maketitle
\begin{abstract}
Linear fractional transformations on
the extended complex plane are
classified up to topological conjugacy.
Recall that two transformations $f$ and
$g$ are called topologically conjugate
if there exists a homeomorphism $h$
such that $g=h^{-1}\circ f\circ h$, in
which $\circ$ is the composition of
mappings.
\end{abstract}

\section{Introduction}

\emph{M\"obius transformations} are linear fractional
transformations of the form
\begin{equation}\label{yrw}
f(z)=\frac{az+b}{cz+d}\,,\qquad ad-bc\ne 0,\quad a,b,c,d\in\mathbb C
\end{equation}
on the extended complex plane $\hat{\mathbb C}:=\mathbb C\cup
\infty$.
The foundations of the theory of
M\"obius transformations are developed in
\cite[Chapters 3 and 4]{Berdon} and \cite[Chapters 8--10]{Mark}.

Since the numbers $a,b,c,d$ can be simultaneously multiplied by any nonzero number without changing $f$, the transformation \eqref{yrw} can be given by the matrix
\begin{equation}\label{feof}
M_f:=\frac{1}{\sqrt{ad-bc}}\begin{bmatrix}
  a & b \\
  c & d \\
\end{bmatrix},
\end{equation}
which has determinant $1$ and is
determined by $f$ uniquely up to multiplication by
$-1$. The composition of
transformations corresponds to the
product of their matrices:
\begin{equation}\label{dfo}
M_{fg}=M_fM_g.
\end{equation}
The trace of a matrix $A$ is denoted by
$\tr A$.

Two
M\"obius transformations $f$ and $g$
are said to be
\begin{itemize}
  \item \emph{conjugate} if there exists a M\"obius transformation $h$ such that
      the diagram
\[
\begin{CD}
\hat{\mathbb C}@>g>>\hat{\mathbb C}\\@VhVV @VVhV
\\\hat{\mathbb C}@>f>>\hat{\mathbb C}
\end{CD}
\] is commutative, i.e., $g=h^{-1}fh$;

  \item \emph{topologically
      conjugate} if there exists a homeomorphism $h:
      \hat{\mathbb C} \to
      \hat{\mathbb C}$ such that $g=h^{-1}fh$ (a mapping $h$ is a
      \emph{homeomorphism} if $h$
      and $h^{-1}$ are continuous
      bijections).
\end{itemize}
If two M\"obius transformations are
conjugate, then they are topologically
conjugate since each M\"obius
transformation is a homeomorphism.

The following criterion of conjugacy is easily obtained from
\cite[Theorem~4.3.1]{Berdon}:
\emph{nonidentity M\"obius
transformations $f$ and $g$ are
conjugate if and only if $\tr M_f=\pm
\tr M_g$}.

If $g=h^{-1}fh$, then $M_g=\pm
M_h^{-1}M_fM_h$ by \eqref{dfo}, and so
conjugate M\"obius transformations are
given by similar matrices
determined up to multiplication by
$-1$. One can take the transforming
matrix $M_h$ such that $M_g$ is the
Jordan form of $M_f$. Since $\det
M_f=\det M_g=1$, \begin{equation}\label{ljkt}
M_g=\begin{bmatrix}
        \lambda  &0 \\
        0 & 1/\lambda \\
      \end{bmatrix}\ (\lambda \ne\pm 1,0) \qquad\text{or}\qquad
M_g=\begin{bmatrix}
        \lambda  &1 \\
        0 & \lambda
      \end{bmatrix} \ (\lambda=\pm 1).
\end{equation}
The matrix $M_g$ is uniquely determined by $f$,
up to multiplication by $-1$ and up to
interchanging the diagonal entries
$\lambda$ and $1/\lambda$ in the first
matrix \eqref{ljkt}. Thus, $f$ is conjugate to
$z\mapsto \lambda^2z$, or $z\mapsto
(1/\lambda^2)z$, or $z\mapsto z+1$, and we obtain a canonical form
for conjugacy \cite[\S 4.3]{Berdon}:
\begin{equation}\label{aim}
\parbox{30em}{\it Each M\"obius transformation is
conjugate
to exactly one M\"obius transformation
of the form
$m_{\mu }(z)= \mu z$ $(\mu\ne 0,1)
$ or $m_{1}(z)= z+1$ $(\mu:=1)$,
in which $\mu $ is determined up to
replacement by $1/\mu$.}
\end{equation}
 The numbers $\mu_1:=\mu $ and
$\mu_2:=1/\mu$ are called the
\emph{multipliers} of $f$. They are defined for a holomorphic map on a Riemann surface in \cite[p. 45]{Milnor}; for a nonidentity M\"obius transformation $f$ they can be calculated by formula
\[
\mu_i=
  \begin{cases}
    f'(z_i) & \text{if } z_i\ne\infty, \\
    \lim\limits_{z\to \infty} \frac{1}{f'(z)} & \text{if } z_i=\infty,
  \end{cases}\qquad i=1,2,
\]
in which $z_1$ and $z_2$ are fixed points of $f$ (their number is 2 or 1; we take $z_1=z_2$ in the latter case).

The main result of this paper is the
following theorem, in which we give 3
criteria of topological conjugacy; the
criterion (iv) was published
in~\cite{Zbir-Inst} in Ukrainian by the first author.

\begin{theorem}\label{klas-drobb}
The following four statements are
equivalent for arbitrary nonidentity
M\"obius transformations
$f,g:\hat{\mathbb C} \to \hat{\mathbb
C}$:
\begin{itemize}
\item[\rm(i)] $f$ and $g$ are
    topologically conjugate;

\item[\rm(ii)] $\tr M_f,\tr
    M_g\notin [-2;2]$ or $\tr
    M_f=\pm\tr M_g$ $($in which
    $[-2;2]$ is the set of all
    $a\in\mathbb R$ satisfying
    $-2\le a\le 2)$;

\item[\rm(iii)] if $\lambda$ is any
    eigenvalue of $M_f$ and
    $\lambda'$ is any eigenvalue of
    $M_g$, then
$|\lambda|,|\lambda'|\ne 1$, or
$\lambda=\pm\lambda'$, or
$\lambda=\pm \bar {\lambda}'$;

\item[\rm(iv)] if $\mu$ is any
    multiplier of $f$ and $\nu$ is
    any multiplier of $g$, then
$|\mu|,|\nu| \ne 1$, or $\mu=\nu$,
    or $\mu=\bar{\nu}$.
\end{itemize}
\end{theorem}

The canonical form \eqref{aim} is used
in the following definition: a
nonidentity M\"obius transformation is
called
\begin{itemize}
  \item \textit{hyperbolic} if it
      is conjugate to $z\mapsto \mu z$ with $1\ne\mu\in\mathbb R$;

  \item \textit{loxodromic} if it
      is conjugate to $z\mapsto \mu
z$ with $\mu\notin\mathbb R$ and
      $|\mu |\ne 1$;

  \item \textit{elliptic} if it
      is conjugate to $z\mapsto \mu
z$ with $|\mu |=1$ and $\mu \ne 1$;

\item \textit{parabolic} if it is
    conjugate to $z\mapsto z+1$.
\end{itemize}

A canonical form of a
M\"obius transformation for topological conjugacy is easily obtained from the equivalence of (i) and (iv)
in Theorem \ref{klas-drobb}:

\begin{corollary}\label{ipyf}
{\rm(a)} Each {hyperbolic} or
{loxodromic} M\"obius transformation is
topologically conjugate to $z\mapsto
2z$.

{\rm(b)} Each elliptic M\"obius
transformation is topologically
conjugate to $z\mapsto \mu z$ $(|\mu|=1)$ which is uniquely
determined up to replacement of $\mu$  by $\bar \mu$.

{\rm(c)} Each parabolic M\"obius
transformation is topologically
conjugate to  $z\mapsto z+1$.
\end{corollary}

Two linear operators $\mathcal A,\mathcal B:\mathbb C^2\to \mathbb C^2$ are said to be {\it topologically
conjugate} if ${\mathcal B}=h^{-1}{\mathcal A}h$ for some homeomorphism $h:
\mathbb C^2 \to \mathbb C^2$.
Assigning to a M\"obius
transformation $f$ the
linear operator $x\mapsto M_fx$
$(x\in\mathbb C^2)$ determined
up to multiplication by $-1$, we obtain the {\it one-to-one correspondence between M\"obius transformations
on $\hat{\mathbb C}$ and linear operators on $\mathbb C^2$ with
determinant 1 that are determined up to multiplication by $-1$}.
This correspondence preserves the
topological conjugacy in virtue of the
following corollary, which will be
proved in Section \ref{kutdx}.

\begin{corollary}\label{lhp}
The following two conditions are
equivalent for M\"obius
transformations $f$ and $g$:
\begin{itemize}
  \item[\rm(i)] $f$ and $g$ are
      topologically conjugate;

  \item[\rm(ii)] the linear
      operator $x\mapsto M_fx$ on
      $\mathbb C^2$ is
      topologically conjugate to
      $x\mapsto M_gx$ or $x\mapsto
      -M_gx$.
\end{itemize}
\end{corollary}

\section{Topological classification
of linear operators}

In this section, we recall some results
of \cite{Kuip-Robb,Robb,Budn-LAA} about
topological classification of linear
operators (they were extended to affine
operators in \cite{bud_ukr,Budn-LAA,bud-bud}), which will be used in the next sections.

For each square complex matrix
$A=[a_{ij}]$, we define the matrix
$\bar A=[\bar a_{ij}]$ whose entries are the complex conjugates of the entries of $A$, and construct a decomposition of $A$ into a direct sum of square matrices
\begin{equation}\label{+-0}
  S^{-1}AS=A_0\oplus A_{01}\oplus A_1\oplus A_{1\infty}\quad\text{($S$ is a nonsingular matrix),}
\end{equation}
in which
all eigenvalues $\lambda$ of
$A_0$ (respectively, of $A_{01}$,
$A_1$, and $A_{1\infty}$) satisfy the
condition
$\lambda =0$ (respectively, $0<|\lambda| <1,$ $|\lambda| =1,$ and $|\lambda| >1$).

The assertion (i) of the following
theorem was proved in
\cite{Kuip-Robb} (see also
\cite{Robb}); the assertion (ii) was
proved  by the first author in \cite[Theorem
2.2]{Budn-LAA}.

\begin{theorem} \label{klas_m1}
Let $f(x)=Ax$ and $g(x)=Bx$ be linear
operators over $\mathbb {F=R}$ or $\mathbb
C$ without eigenvalues that are roots
of unity, and let
$A_0,A_{01}, A_1, A_{1\infty}$ and
$B_0,B_{01}, B_1, B_{1\infty}$ be constructed
by $A$ and $B$ as in \eqref{+-0}.

\begin{itemize}
  \item[{\rm(i)}] If $\mathbb
      {F=R}$, then $f$ and $g$ are
      topologically conjugate if
      and only if
\begin{equation}\label{topol-matr-r}
\!\!\!\!\!\!\!\!\begin{matrix}
A_0\text{ is similar to } B_0,\ \
\size A_{01} =\size B_{01}, \ \ \det (A_{01} B_{01})>0,\\
A_1\text{ is similar to } B_1,\ \
\size A_{1\infty} =\size B_{1\infty}, \ \ \det (A_{1\infty} B_{1\infty})>0.
\end{matrix}\!\!\!\!\!
\end{equation}

  \item[{\rm(ii)}] If $\mathbb
      {F=C}$, then $f$ and $g$ are
      topologically conjugate if
      and only if
      \begin{equation*}\label{topol-matr-C1}
\begin{matrix}
A_0 \text{ is similar to }
B_0,\quad
\size A_{01} =\size B_{01}, \\
A_1\oplus \bar A_1 \text{ is
similar to } B_1\oplus \bar
B_1,\quad \size A_{1\infty} =\size
B_{1\infty}.
\end{matrix}
\end{equation*}
\end{itemize}
\end{theorem}

Two linear operators $f$, $g: \mathbb
R^n \to \mathbb R^n$ are said to be
\emph{conjugate} if there is a
linear bijection  $h: \mathbb
R^n \to \mathbb R^n$ such that
$g=h^{-1}f h$.

An operator $f$ is \emph{periodic} if
$f^k$ is the identity for some natural
$k$. Kuiper and Robbin~\cite{Kuip-Robb} proved that if the hypothesis
\begin{equation}\label{4kj8}
\parbox{25em}
{two periodic linear operators are
topologically conjugate if and only if
they are conjugate}
\end{equation}
is true, then \eqref{topol-matr-r} are necessary
and sufficient conditions of
topological conjugacy for all linear
operators.
Cappell and
Shaneson~\cite{Capp-2th-nas-n<=6,Capp-big-n<6,
Capp-Contemp-n<6,Capp-n=6} proved \eqref{4kj8}
for all linear operators on $\mathbb R^n$
with $n< 6$.

These results ensure the assertion (i)
of the following theorem. The assertion
(ii) is proved as Theorem
2.2 in \cite{Budn-LAA}.

\begin{theorem}\label{klas_m}
Let $f(x)=Ax$ and $g(x)=Bx$ be linear
operators on $V=\mathbb
      R^{m}$ or $\mathbb
      C^{m}$; let
$A_0,A_{01}, A_1, A_{1\infty}$ and
$B_0,B_{01}, B_1, B_{1\infty}$ be constructed
by $A$ and $B$ as in \eqref{+-0}.

\begin{itemize}
  \item[{\rm(i)}] Suppose
      $V=\mathbb R^{m}$ with $m\le
      5$; then $f$ and $g$ are
      topologically conjugate if
      and only if
\begin{equation*}\label{topol-matr-R}
\begin{matrix}
A_0 \text{ is similar to } B_0,\quad
\size A_{01} =\size B_{01}, \quad \det (A_{01} B_{01})>0,\\
A_1 \text{ is similar to } B_1,\quad
\size A_{1\infty} =\size B_{1\infty}, \quad \det (A_{1\infty} B_{1\infty})>0.
\end{matrix}
\end{equation*}

 \item[{\rm(ii)}] Suppose
     $V=\mathbb C^m$ with $m\le 2$;
     then $f$ and $g$ are
     topologically conjugate if and
     only if
\begin{equation}\label{oibi}
\begin{matrix}
A_0 \text{ is similar to }
B_0,\quad
\size A_{01} =\size B_{01}, \\
A_1\oplus \bar A_1 \text{ is
similar to } B_1\oplus \bar
B_1,\quad \size A_{1\infty} =\size
B_{1\infty}.
\end{matrix}
\end{equation}
\end{itemize}
\end{theorem}

The following corollary is used in Section \ref{kutdx} for the
proof of Corollary \ref{lhp}.

\begin{corollary}\label{poit}
Let $f(x)=Ax$ and $g(x)=Bx$ be two
nonidentity linear operators on
$\mathbb C^2$ whose matrices $A$ and
$B$ have determinant $1$ and are
diagonalizable $($i.e., their Jordan
forms are diagonal$)$. Let $\lambda$
and $\lambda'$ be eigenvalues of $A$
and $B$, respectively. Then $f$ and $g$
are topologically conjugate if and only
if $|\lambda|,|\lambda'| \neq 1$, or
$\lambda=\lambda',$ or $\lambda=\bar
{\lambda}'. $
\end{corollary}

\begin{proof} The conjugacy
of linear operators implies their
topological conjugacy, so we can
suppose that the matrices of $f(x)=Ax$
and $g(x)=Bx$ are given in their Jordan
forms:
$$
A=\begin{bmatrix}
\lambda  & 0 \\
0 & \lambda ^{-1}
\end{bmatrix},\quad
B=\begin{bmatrix}
\lambda'  & 0 \\
0 & \lambda^{\prime -1}
\end{bmatrix},
\qquad \lambda,\lambda' \ne \pm1,0.
$$

The following 4 cases are possible:
\begin{description}
  \item[\it Case 1: $|\lambda|\ne
      1$ and $|\lambda'|\ne 1$.]
      Then in notation \eqref{+-0}
      \[{A}_{01}\oplus{A}_{1\infty}=
      [\lambda]\oplus[\lambda^{-1}],\qquad
      {B}_{01}\oplus{B}_{1\infty}=
      [\lambda']\oplus[\lambda'^{-1}].\]
      By Theorem~\ref{klas_m}(ii),
      $f$ and $g$ are topologically
      conjugate.

   \item[\it Case 2:
       $|\lambda|=|\lambda'|=1$.]
       Then \[{A}_{1}\oplus \bar
       {A}_{1}=[\lambda]\oplus
[\bar\lambda],\qquad {B}_{1}\oplus
\bar {B}_{1}=[\lambda']\oplus
[\bar{\lambda}'].\] By
Theorem~\ref{klas_m}(ii), $f$ and
$g$ are topologically conjugate if
and only if $\lambda=\lambda'$ or
$\lambda=\bar {\lambda}'$.

 \item[\it Case 3: $|\lambda|=1$
     and $|\lambda'|\ne 1$.] Then
      \[
{A}_{1}\oplus \bar
{A}_{1}=[\lambda]\oplus
[\bar\lambda],\qquad
{B}_{01}\oplus{B}_{1\infty}=
[\lambda']\oplus[\lambda'^{-1}].\]
By Theorem~\ref{klas_m}(ii), $f$
and $g$ are not topologically
conjugate.

 \item[\it Case 4: $|\lambda|\ne 1$
     and $|\lambda'|= 1$.] Then $f$
     and
$g$ are not topologically
conjugate.
\end{description}
\end{proof}

\begin{lemma}\label{kl}
M\"obius
transformations $f(z)=az$ and $g(z)=bz$ are topologically
conjugate if and only if
\begin{equation}\label{dwj}
\text{$|a|,|b|\ne 1$,\quad or
$a=b$,\quad or
$a=\bar{b}$.}
\end{equation}
\end{lemma}

\begin{proof} $\Longleftarrow$. Suppose that $f$
and $g$ satisfy \eqref{dwj}.

If $f$ and $g$ satisfy
\begin{equation}\label{dsy}
|a|,|b|<1,
\quad\text{or }|a|,|b|>1,
\quad\text{or } a=b,
\quad\text{or } a=\bar b,
\end{equation}
then by Theorem~\ref{klas_m}(ii) the
linear mappings $z\mapsto az$ and $z
\mapsto bz$ on $\mathbb C$ are
topologically conjugate via some
homeomorphism $\eta: \mathbb C  \to
\mathbb C$, and so $f$ and $g$ are
topologically conjugate via the
homeomorphism $h:\hat{\mathbb C} \to
\hat{\mathbb C}$ defined as follows:
$h(z):=\eta(z)$ if $z \in \mathbb C$
and $h(\infty):=\infty$.

If $f$ and $g$ do not satisfy
\eqref{dsy} (but satisfy \eqref{dwj}),
then either $|a|<1$ and $|b|>1$, or
$|a|>1$ and $|b|<1$. Suppose that
$|a|<1$ and $|b|>1$. Then $|1/b|<1$ and
by \eqref{dsy} $f$ is topologically
conjugate to $g^{-1}(z)=(1/b)z$, which
is topologically conjugate to $g$ via
the homeomorphism $z\mapsto 1/{z}$ on
$\hat{\mathbb C}$.

$\Longrightarrow$. Let M\"obius
transformations $f(z)=az$ and $g(z)=bz$
on $\hat{\mathbb C}$ be topologically
conjugate; that is, there exists a
homeomorphism ${h}:\hat{\mathbb C} \to
\hat{\mathbb C}$ such that
\begin{equation}\label{tdf}
{h}g(z)= f {h}(z)\qquad \text{for all }z\in \hat{\mathbb C}.
\end{equation}
Since ${h}$ transforms all fixed points
of $g$ to all fixed points of $f$ and
their fixed points  are  $0$ and
$\infty$, the following two cases are
possible.

\begin{description}
 \item[\it Case 1:
     ${h}(\infty)=\infty$ and
${h}(0)=0$.] By~\eqref{tdf}, the
     linear operators $z\mapsto az$
     and $z \mapsto bz$ on $\mathbb
     C$ (which are the restrictions
     of $f$ and $g$ to $\mathbb C$)
     are topologically conjugate
     via the homeomorphism that is
     the restriction of $h$ to
     $\mathbb C$.
     Theorem~\ref{klas_m}(ii)
     ensures that
\begin{equation}\label{dgj}
|a|,|b|<1,\quad\text{or }
|a|,|b|>1,\quad
\text{or } a=b,\quad
\text{or } a=\bar b.
\end{equation}

 \item[\it Case 2: ${h}(\infty)=0$
     and ${h}(0)=\infty$.] The
     M\"obius transformations
     $f^{-1}(z)=(1/a)z$ and
     $g(z)=bz$ are topologically
     conjugate via
     ${h_1}:=\varphi{h}$ in which
     $\varphi(z):=1/z$. Since
     ${h}_1(\infty)=\infty$, we
     have \eqref{dgj} in which $a$
     is replaced by $1/a$.
\end{description}
In both the cases, $a$ and $b$ satisfy
\eqref{dwj}.
\end{proof}

\section{Proof of Theorem~\ref{klas-drobb}}
\label{s3}


In this section, we denote by $f$ and
$g$ two nonidentity M\"obius
transformations, by $\lambda$ and
$\lambda'$ arbitrary eigenvalues of
$M_f$ and $M_g$, respectively, and by
\begin{equation}\label{lir}
\text{$n(f)$ and
$n(g)$ the numbers of fixed points of $f$ and $g$.}
\end{equation}
Recall that the number of fixed points of
any nonidentity M\"obius transformation
is equal to 1 or 2.

(i) $\Longleftrightarrow $ (iv).
Suppose (i) holds. Then $n(f)=n(g)$.
Since $f$ and $g$ are not the identity,
the following two cases are possible.
\begin{description}
\item[{\it Case 1: $n(f)=n(g)=1$.}]
    By \eqref{aim}, $f$ and $g$ are
    conjugate to $m_1(z)= z+1$,
    whose multiplier is $1$. This
    ensures (iv).

\item[{\it Case 2: $n(f)=n(g)=2$.}]
    Let $\mu,\nu\notin\{0,1\}$ be
    multipliers of $f$ and $g$,
    respectively. By \eqref{aim},
    $f$ and $g$ are conjugate to
    $m_{\mu}(z)=\mu z$ and
    $m_{\nu}(z)=\nu z$, whose fixed
    points are $0$ and $\infty$.
    Lemma~\ref{kl} ensures (iv).

\end{description}
Thus $\rm(i)\Longrightarrow \rm(iv)$.
The converse arguments give
$\rm(i)\Longleftarrow \rm(iv)$.


$\rm(iii) \Longleftrightarrow\rm(iv)$.
By \eqref{aim} and \eqref{feof}, if
$\mu$ is a multiplier of $f$, then $f$
is conjugate to $m_{\mu}(z)=\mu z$
$(\mu\neq 0,1)$ or $m_{1}(z)= z+1$
$(\mu=1)$, and so $M_f$ is similar to
\[
M_{m_{\mu}}=\pm\frac 1{\sqrt{\mu}}
\begin{bmatrix}
\mu & 0 \\
 0  & 1 \\
\end{bmatrix}
=\pm\begin{bmatrix}
\sqrt{\mu} & 0 \\
 0         & 1/\sqrt{\mu}
\end{bmatrix}\quad (\mu \ne 0,1)
\]
or
\begin{equation*}\label{mcv}
M_{m_{\mu}}= \pm
\begin{bmatrix}
 1 & 1 \\
 0 & 1 \\
\end{bmatrix}\quad (\mu =1).
\end{equation*}

Therefore, the matrix $M_f$ has  an eigenvalue $\lambda$ that is equal to
\begin{equation}\label{sss}
\sqrt{\mu}\ \ \text{ or }\  -\sqrt{\mu} \quad (\mu \ne 0).
\end{equation}
Analogously, if $\nu$ is a multiplier
of $g$, then $M_g$ has an eigenvalue
$\lambda'$ that is equal to
$\sqrt{\nu}$ or $-\sqrt{\nu}\ (\nu \ne
0)$, which proves the equivalence of
(iii) and (iv).


$\rm(ii) \Longleftrightarrow\rm(iii)$.
Let us prove that
\begin{equation}\label{ajgm}
\parbox{20em}{$|\lambda|=1 \quad
\Longleftrightarrow \quad \tr M_f=\pm \tr M_{m_\mu}
\in [-2;2]$.}
\end{equation}
The equality $\pm\tr M_f=\tr M_{m_\mu}$
follows from the similarity of $M_f$ and
$M_{m_\mu}$. If $|\lambda|=1$, then by
\eqref{ljkt}
\[
\tr M_{m_\mu}=\lambda +\lambda^{-1}=
\lambda +\bar\lambda\in [-2;2].
\]
If $|\lambda|\ne 1$, then
$\lambda^{-1}\ne\bar\lambda$ and $\tr
M_{m_\mu}=\lambda +\lambda^{-1}\notin
[-2;2]$, which proves \eqref{ajgm}. The
following 3 cases are possible.

\begin{description}
\item[{\it Case 1: $|\lambda|\ne 1$
    and $|\lambda'|\ne1$.}] Then
 (iii) holds. By
    \eqref{ajgm}, (ii) holds too.

 \item[\it Case 2: $|\lambda|=1$
     and $|\lambda'|\ne 1$, or
     $|\lambda|\ne 1$ and
     $|\lambda'|=1$.] Then (ii) and
     (iii) do not hold.

\item[{\it Case 3:
    $|\lambda|=|\lambda'|=1$.}] The
    condition $\tr M_f=\pm\tr M_g$
    is equivalent to $\tr
    M_{m_\mu}=\pm \tr M_{m_\nu}$ is
    equivalent to $\lambda
    +\lambda^{-1}=\pm (\lambda'
    +\lambda^{\prime-1})$ is
    equivalent to $\lambda
    +\bar\lambda=\pm (\lambda'
    +\bar\lambda^{\prime})$ is
    equivalent to
    $\lambda=\pm\lambda'$ or
    $\lambda=\pm \bar {\lambda}'$.
\end{description}

\section{Proof of Corollary~\ref{lhp}}\label{kutdx}

The following 4 cases are possible for
any M\"obius transformations $f$
and $g$.

\begin{description}
\item[\it Case 1: $n(f)\ne n(g)$]
    $($see \eqref{lir}$)$. Then
    the assertion (i) of
    Corollary~\ref{lhp} does not
    hold. Let us prove that (ii)
    does not hold too. Suppose that
    $n(f)<n(g)$. If $n(g)=\infty$,
then $g$ is the identity,
$n(f)\in\{1,2\}$, and (ii) does not
hold. Suppose that $n(g)<\infty$.
Then $n(f)=1$ and $n(g)=2$.
By~\eqref{aim} and \eqref{sss}, $f$
is conjugate to $m_1(z)=z+1$ and
$g$ is conjugate to $m_{\mu}(z)=
\lambda^2z$. The linear operators
$x\mapsto M_fx$ and $x\mapsto M_gx$
are conjugate to $x\mapsto\pm
M_{m_1}x$ and $x\mapsto \pm
M_{m_{\mu}}x$, respectively, in
which
$$
M_{m_1}=
\begin{bmatrix}
1 & 1 \\
0 & 1
\end{bmatrix}, \qquad
M_{m_{\mu}}=
\begin{bmatrix}
\lambda & 0 \\
0 & {\lambda}^{-1}
\end{bmatrix}\ (\lambda \ne 0,1).
$$
The vector $[0,0]^T$ is the only
fixed point of the linear operator
$x\mapsto M_{m_{\mu}}x$. All
vectors $[a,0]^T$
($a\in\mathbb C$) are fixed points of
$x\mapsto M_{m_1}x$. Thus, the
assertion (ii) does not hold.

\item[{\it Case 2: $n(f)= n(g)=1$.}]
    By~\eqref{aim}, $f$ and $g$ are
    conjugate to $z\mapsto z+1$.
    By~\eqref{ljkt}, the matrices $M_f$
    and $M_g$ are similar to
$$
\begin{array}{l}
\begin{bmatrix}
1 & 1 \\
0 & 1
\end{bmatrix}
\quad\text{or}\quad
\begin{bmatrix}
-1 & 1 \\
0 & -1
\end{bmatrix}.
\end{array}
$$
Hence $M_f$ is similar to $M_g$ or
$-M_g$. Therefore, $x\mapsto M_fx$
is conjugate to $x\mapsto
\pm M_gx$, and so they are
topologically conjugate.

\item[{\it Case 3: $n(f)=
    n(g)=2$.}] By~\eqref{aim} and
    \eqref{sss}, $f$ is conjugate
    to $z\mapsto \lambda^2z$ and
    $g$ is conjugate to $z\mapsto
    {\lambda'}^2z$, in which
    $\lambda$ and $\lambda'$ are
    eigenvalues of $M_f$ and $M_g$,
    respectively. The Jordan forms
    of $M_f$ and $M_g$ are $\pm
    J_f$ and $\pm J_g$, where
\[
J_f:=
\begin{bmatrix}
\lambda & 0 \\
0 & \lambda^{-1}
\end{bmatrix},
\quad
J_g:=
\begin{bmatrix}
\lambda' & 0 \\
0 & {\lambda'}^{-1}
\end{bmatrix},\qquad \lambda,
\lambda'\notin\{0,\pm
1\}.
\]

By Theorem~\ref{klas-drobb}(iii),
$f$ and $g$ are topologically
conjugate if and only if
\begin{equation*}\label{cf}
|\lambda|,|\lambda'| \neq 1,
\quad\text{or}\quad\lambda=\pm \lambda',
\quad\text{or}\quad\lambda=
\pm \bar {\lambda}',
\end{equation*}
if and only if (see Corollary
\ref{poit}) the linear operators
$x\mapsto J_fx$ and $x\mapsto \pm
J_gx$ are topologically conjugate,
if and only if the linear operators
$x\mapsto M_fx$ and $x\mapsto \pm
M_gx$ are topologically conjugate.

\item[{\it Case 4: $n(f)=
    n(g)>2$.}] The assertions (i)
    and (ii) hold since $f$ and $g$
    are the identity mappings and
    $M_f=\pm M_g=\pm I_2$, in which
    $I_2$ is the identity matrix.
\end{description}

\end{document}